\documentclass[a4wide, 11pt]{article}
\usepackage[english]{babel}
\usepackage{a4, amsthm}
\usepackage{amsmath}
\usepackage{amssymb}
\usepackage[margin=0.9in]{geometry}
\usepackage{todonotes}

\usepackage{setspace}
\makeatletter
\def\thanks#1{\protected@xdef\@thanks{\@thanks
        \protect\footnotetext{#1}}}
\makeatother

\theoremstyle{definition}
\newtheorem{definition}{Definition}

\theoremstyle{plain}
\newtheorem{theorem}[definition]{Theorem} 
\newtheorem{lemma}[definition]{Lemma}
\newtheorem{Proposition}[definition]{Proposition}
\newtheorem{Corollary}[definition]{Corollary}

\DeclareMathOperator{\F}{\mathbb{F}}
\DeclareMathOperator{\GL}{GL}
\DeclareMathOperator{\GamL}{\Gamma L}

\DeclareMathOperator{\AGL}{AGL}

\DeclareMathOperator{\Sp}{Sp}

\DeclareMathOperator{\Sym}{Sym}
\newcommand{\Sn}{\mathrm{S}}

\newcommand{\mC}{\mathcal{C}}
\newcommand{\mO}{\Xi}
\newcommand{\mG}{\mathcal{G}}
\newcommand{\mB}{\mathcal{B}}

\numberwithin{definition}{section}
\numberwithin{equation}{section}

\title{Irredundant bases for soluble groups\thanks{The authors are grateful to the anonymous referees for their careful reading and comments which have improved the paper. Sofia Brenner thanks the University of St Andrews for the hospitality during a visit in July 2024, where part of this work was undertaken. She received funding from the European Research
Council (ERC) under the European Union’s Horizon 2020 research and innovation
programme (EngageS: grant agreement No.~820148) and from the German Research
Foundation DFG (SFB-TRR 195 ``Symbolic Tools in Mathematics and their Application'' as well as grant 522790373).
Coen del~Valle is supported by the Natural Sciences and Engineering Research Council of Canada (NSERC), [reference number PGSD-577816-2023], as well as a University of St Andrews School of Mathematics and Statistics Scholarship.}}
\author{Sofia Brenner, Coen del Valle, and Colva M. Roney-Dougal}

\date{\today}

\begin{document}
\maketitle
\begin{abstract}
    Let $\Delta$ be a finite set and $G$ be a subgroup of $\Sym(\Delta)$. An irredundant base for $G$ is a sequence of points of $\Delta$ yielding a strictly descending chain of pointwise stabilisers, terminating with the trivial group. Suppose that $G$ is primitive and soluble. We determine asymptotically tight bounds for the maximum length of an irredundant base for $G$. Moreover, we disprove a conjecture of Seress on the maximum length of an irredundant base constructed by the natural greedy algorithm, and prove Cameron's Greedy Conjecture for $|G|$ odd.

    \bigskip

    
\end{abstract}
\section{Introduction}

We analyse subgroups $G$ of $\Sym(\Delta)$, with $\Delta$ finite. 
A sequence $(\delta_1,\delta_2, \dots, \delta_k)\in\Delta^k$ is \emph{irredundant} for~$G$ if
\[
G > G_{\delta_1} > G_{(\delta_1, \delta_2)} > \dots > G_{(\delta_1, \dots, \delta_k)},
\]
and is a \emph{base} if $G_{(\delta_1, \dots, \delta_k)} = 1$.
Define $b(G)$ and $I(G)$ to be the minimum and maximum $k$, respectively, such that there exists an irredundant base for $G$ of length $k$.

Of great importance in computational group theory, the statistic $b(G)$ has been studied extensively over the past 40 years, especially for $G$  primitive --- see~\cite{Maroti23} for a review of recent work.
In particular, in 1996 Seress~\cite{Seress96} showed that every primitive soluble permutation group has a base of size at most~4, a result which was recently extended by Burness~\cite{Burness} to $G$ primitive with soluble point stabilisers, showing that $b(G)\leq 5$.
The statistic $I(G)$ is much less understood. 
It is a straightforward calculation to show that $I(G)\leq b(G)\log|\Delta|$ (all logarithms will be base 2), but recent work~\cite{Gill_Loda_Spiga_2022,Kelsey_Roney-Dougal} has shown that in fact $I(G)$ is bounded above by a constant multiple of $\log|\Delta|$ whenever $G$ is a small base primitive group. Stronger bounds have been determined for some almost simple primitive groups~\cite{Gill_Liebeck, Roney-Dougal_Wu}. 

In this paper, we first study  $I(G)$ for $G$ primitive and soluble. 
 Here  $G$ is a subgroup of $\AGL_d(p)$ for some prime $p$, hence $\Delta$ can be identified with the points of $\F_p^d$ so that the point stabiliser $H:= G_0$ is an irreducible subgroup of $\GL_d(p)$. 
We shall in fact consider the slightly more general setting of $H$ an irreducible soluble subgroup of $\GL_d(q)$ 
for an arbitrary prime power~$q$, with $\Delta = \F_{q}^d$.  
The irreducible group $H$ is \emph{imprimitive} if $H$ preserves a direct sum decomposition $\F_q^d = V_1 \oplus \dots \oplus V_k$ of non-zero subspaces for some $k \geq 2$, and otherwise $H$ is \emph{primitive}.

For any subgroup of $\GL_d(q)$, every irredundant base is a linearly independent subset of $\F_q^d$, so the bound $I(H) \leq d$ is immediate. Our first main theorem strengthens this observation.  Throughout this paper, $\Omega(n)$ denotes the number of prime divisors of $n$, counted with multiplicities.


\begin{theorem}\label{theo:main}
Let $H$ be an irreducible soluble subgroup of $\GL_d(q)$, acting on the vectors of $\F_q^d$. Then $I(H) \leq d$, and  for every $d$ and every $q > 2$ there exist such groups $H$ with $I(H) = d$.

Furthermore, if $H$ is primitive as a linear group then $I(H) \leq  6.49 \log d + 1$. 
For every $d$ and every $q$ there exists such a group $H$ with $I(H) \geq \Omega(d)+1$, and every primitive $H$ that is not semilinear and has an absolutely irreducible Fitting subgroup satisfies $I(H) \ge \Omega(d)$. 
\end{theorem}
If $H$ is imprimitive, 
then the corresponding group $G$ of affine type is permutation isomorphic to  a subgroup of a wreath product group with product action, and is said to be \emph{non-basic}. Conversely if $H$ is primitive then $G$ preserves no such product action structure, and is \emph{basic}. We deduce the following by noting that $d \le \log |\Delta|$, and $d > 0.63 \log |\Delta|$ if $q > 2$.

\begin{Corollary}
Let $G$ be a primitive soluble permutation group of degree $n$. 
Then
$I(G) \le \log n + 1$, and there exist infinitely many such groups $G$ for which $I(G) \ge 0.63 \log n + 1$. Furthermore, if $G$ is basic then $I(G) \le  6.49 \log \log n + 2$, and there exist infinitely many such groups $G$ for which  $I(G) \geq \log \log n +2$. 
\end{Corollary}

We go on to study the \emph{greedy base size} $\mathcal{G}(G)$. This is the maximum size of a base for $G$ produced by Blaha's greedy algorithm~\cite{Blaha92}, as follows. We start with an empty sequence  
$\mathcal{B}_0$.  At step $i \ge 1$, we compute the pointwise stabiliser $G^{i-1}$ of the current sequence $\mathcal{B}_{i-1}$, and let the new sequence $\mathcal{B}_i$ be $\mathcal{B}_{i-1}$ extended by any point from a longest orbit of $G^{i-1}$. We stop when $G^i$ becomes trivial, and call $\mathcal{B}_i$  a \emph{greedy base}.
In~\cite{Blaha92}, Blaha shows that there is some absolute constant $d$ such that $$\mathcal{G}(G)\leq db(G)\log \log |\Delta|,$$
and moreover, for all $k\geq 2$ 
 he constructs a group $G$ of degree $n$ satisfying $b(G)=k$ and $\mathcal{G}(G)\geq\frac{1}{5}k\log\log n$.
%
 These groups constructed by Blaha are intransitive, and in 1999 Peter Cameron made his Greedy Conjecture~\cite{CAM99}: there is an absolute constant $c$ such that for all primitive $G$, $$\mathcal{G}(G)\leq cb(G).$$
Very recently, Cameron's Greedy Conjecture has been studied for
almost simple primitive groups with socle either sporadic~\cite{delValle24} or alternating~\cite{delValle_Roney-Dougal}. 
Our next main result proves Cameron's Greedy Conjecture for groups of odd order, in its strongest possible form. 

\begin{theorem}\label{oddgreedy}
Let $G$ be primitive and of odd order. Then $b(G)=\mathcal{G}(G)$.
\end{theorem}
In 1996, Seress proved  that $b(G)\leq 3$ for $G$ primitive of odd order \cite[Theorem 1.3]{Seress96}, and conjectured that 
the greedy base size of every primitive soluble group is bounded above by 4.
Our final main result disproves this conjecture.
\begin{theorem}\label{counterexample}
There is a soluble primitive group $G$ with $\mathcal{G}(G)>4$.
\end{theorem}
 In Section~\ref{sec:primitive} we  prove Theorem~\ref{theo:main}, then in Section~\ref{Greedy} we prove Theorems~\ref{oddgreedy} and~\ref{counterexample}.

\section{Irredundant bases of soluble linear groups}\label{sec:primitive}
In this section we prove Theorem~\ref{theo:main}. Throughout, we let $H$ be a soluble subgroup of $\GL_d(q)$ and~$\Delta$ be the set of 
vectors of $\F_{q}^{d}$, and we consider $I(H)$ for the action of $H$ on $\Delta$.  
In Subsection~\ref{subsec:c6} we 
bound $I(H)$ when
$H$ is primitive and normalises an absolutely irreducible group of extraspecial type. 
In Subsection~\ref{subsec:tensorinduced}, we consider tensor products of these groups,
then we prove Theorem~\ref{theo:main} in Subsection~\ref{sec:proofmain1}. For a comprehensive reference on the structure of maximal soluble subgroups of $\GL_d(q)$, see \cite{KorhonenBook}.

First, we establish two basic facts about irredundant bases.
The maximum length~$l$ of a strictly descending chain $G = G_0 > G_1> \dots > G_l = 1$ of subgroups of a finite group $G$ will be denoted $\ell(G)$. Clearly, $\ell(G) \leq \Omega(|G|)\leq \log |G|$. 
The first claim of the following lemma is clear, for the second see \cite[Lemma~2.8]{Gill_Loda_Spiga_2022}.

\begin{lemma}\label{lem:boundnormalsubgroup}
Let $G$ be a finite permutation group, let $S$ be a subgroup of $G$ and let $N$ be a normal subgroup of $G$. Then
$I(S) \leq I(G)$, and $I(G) \leq I(N) + \ell(G/N)$. 
\end{lemma}


\subsection{Normalisers of groups of extraspecial type}\label{subsec:c6}

In this subsection, we suppose that $H$ normalises an absolutely irreducible group of extraspecial type and bound $I(H)$: this will be a base case later on.
We first recall some facts about extraspecial groups. 

Let $r$ be prime. An $r$-group $E$ is \emph{extraspecial} if $\Phi(E) = Z(E) = E' \cong C_r$, where $\Phi(E)$ denotes the Frattini subgroup of $E$. 
For each $m \ge 1$ there exist two non-isomorphic extraspecial groups of order $r^{1+2m}$, denoted by $r_+^{1+2m}$ and $r_-^{1+2m}$. If $r$ is odd then the former has exponent $r$ and the latter exponent $r^2$, whilst
if $r = 2$ then
both 
have exponent four, but  $2^{1+2m}_+ \cong D_8 \circ \dots \circ D_8$ (with $m$ factors $D_8$) and $2^{1+2m}_- \cong Q_8 \circ D_8 \circ \dots \circ D_8$ (with $m-1$ factors $D_8$). It is well known that $C_4 \circ 2^{1+2m}_+ \cong C_4 \circ 2^{1+2m}_-$ (for instance, see \cite[Proposition~2.4.9]{SHO92}), and a group~$E$ of this form is called a \emph{2-group of symplectic type} and denoted $2^{2+2m}$. 
We shall refer to the extraspecial $r$-groups and the 2-groups of symplectic type as \emph{$r$-groups of extraspecial type}. 

Moving on now to representations of these groups, we say that a group $L \leq \GL_{d}(q)$ is in \emph{Class $\mathcal{C}_6$} 
if~$L$ contains an absolutely irreducible normal subgroup $E$ of extraspecial type, of exponent~$r$ when~$r$ is odd. Such an $L$ exists if and only if $d = r^m$ and $r \mid q-1$, and if $E$ is of symplectic type then $4 \mid q-1$. Let $Z$ denote the centre of $\GL_{r^m}(q)$ 
and $K = EZ$. 
 Then 
$\GL_{r^m}(q)$ contains a unique class of absolutely irreducible groups isomorphic to $E$ (see \cite[Theorem~2.4.7]{SHO92}) and $LK/K$ is a subgroup of 
\[
N_{\GL_{r^m}(q)}(K) /K \cong  \begin{cases}
  \Sp_{2m}(r) & \text{if } r > 2 \mbox{ or }  E\cong 2^{2+2m}\\
\mathrm{O}_{2m}^\epsilon(2) & \text{if } r = 2 \text{ and } E \cong 2_\epsilon^{1+2m}, \text{ where } \epsilon \in \{+, -\} 
\end{cases}
\]
(see \cite[Theorem~2.4.12]{SHO92}).  For $\epsilon \in \{+, -\}$ the groups $\mathrm{O}_{2m}^{\epsilon}(2)$  are subgroups of $\Sp_{2m}(2)$, so  $2_{\epsilon}^{1+2m}.\mathrm{O}_{2m}^\epsilon(2)$ is a subgroup of $2^{2+2m}.\Sp_{2m}(2)$. If  $L$ is primitive then $LK/K$ 
is a completely reducible subgroup of $\Sp_{2m}(r)$. 

We now derive an upper bound on $I(H)$, using the following bound due to P\'alfy \cite{Palfy82} and Wolf \cite{Wolf82}
\begin{lemma}\label{lemma:sizeboundsp}
Let $R \le \GL_d(q)$ be completely reducible and soluble. Then $|R| \leq 24^{-1/3} q^{2.244 d} < q^{2.244d}$. 
\end{lemma}

Our next result applies in particular to all primitive soluble Class $\mathcal{C}_6$ groups. 

\begin{Proposition}\label{prop:boundextraspecialcase}
Suppose that $H \leq \GL_{r^m}(q)$ is soluble, in Class $\mathcal{C}_6$, and acts completely reducibly on the normal subgroup $E$ of extraspecial type. Then 
\[|H| \leq |H \cap Z|r^{6.49 m}  \leq (q-1) r^{6.49 m} \quad  \mbox{ and } \quad I(H) \leq 1 + 6.49 \log r^m = 1 + 6.49 \log d.\]  
\end{Proposition}

\begin{proof}
Let $L = (H \cap Z) E$.
Then by Lemma~\ref{lem:boundnormalsubgroup}, 
\[
I(L) \leq I(H \cap Z) + \ell\left(L/(H \cap Z)\right) =  1 + \ell(E/(E \cap Z)) =  1 + 2m 
\]
as $H \cap Z$ is a group of scalar matrices. 
Since $H/L$ is a completely reducible soluble subgroup of $\Sp_{2m}(r)$, Lemma~\ref{lemma:sizeboundsp} yields $|H/L| \leq r^{2 \cdot 2.244 m}  < r^{4.49 m}$. We therefore deduce from Lemma~\ref{lem:boundnormalsubgroup} that 
\[
I(H) \leq I(L) + \ell(H/L) \leq 1+2m  + {4.49} m \log r \leq 1 + 6.49 \log r^m.
\]
Finally,  $|E/(E \cap Z)| = r^{2m}$ 
gives  $|H| \leq |H \cap Z| \,r^{6.49 m}$.
\end{proof}

We now show that there exist soluble groups $H$ in Class $\mathcal{C}_6$ for which $I(H)$ attains, up to constants, the bound given in Proposition~\ref{prop:boundextraspecialcase}. To this end, we use explicit representations of these groups: see \cite{SHO92} or \cite{Holt_Roney-Dougal_2005}. 
Let $\omega$ denote a primitive $r$-th  root of unity in $\F_q$, let $x = \operatorname{diag}\left(1,\omega, \omega^2, \dots, \omega^{r-1}\right) \in \GL_r(q)$ and let $y$ be the permutation matrix in $\GL_r(q)$ corresponding to the $r$-cycle $(1, 2, \ldots, r)$. 
Let $I_t$ denote the identity of $\GL_t(q)$, and for $i \in \{1, \dots, m\}$, let $x_i = I_{r^{m-i}} \otimes x \otimes I_{r^{i-1}} \in \GL_{r^m}(q)$ and $y_i = I_{r^{m-i}} \otimes y \otimes I_{r^{i-1}} \in \GL_{r^m}(q)$. 
If $q \equiv 1 \pmod{4}$, let $\zeta$ denote a primitive element of $\F_q$ and set $z = \zeta^{(q-1)/4} I_{r^m}$. For $r = 2$, fix any $\alpha, \beta \in \F_q$ with $\alpha^2 + \beta^2 = -1$, and let
\[
x_1' = I_{2^{m-1}} \otimes \begin{pmatrix}
\alpha & \beta \\
\beta & -\alpha
\end{pmatrix}
 \quad \mbox { and } \quad 
y_1' = I_{2^{m-1}} \otimes \begin{pmatrix}
  0 & -1 \\
  1 & 0
\end{pmatrix}.
\]

\begin{lemma}[{\cite[Lemmas~9.1, 9.3 and 9.4]{Holt_Roney-Dougal_2005}}]\label{lemma:propxi}
The following hold for the matrices defined above. 
\begin{enumerate}
    \item For $i \in \{1, \dots, m\}$, the group $\langle x_i, y_i \rangle \cong r_+^{1+2}$, and if $i \neq j$, then $[x_i, x_j] = [y_i, y_j] = [x_i, y_j] = 1$. 
    \item The group $\langle x_1', y_1'\rangle \cong Q_8$, and $[x_1', x_i] = [y_1', x_i] = [x_1', y_i] = [y_1', y_i] = 1$ for $i \in \{ 2, \dots, m\}$. 
\end{enumerate}
\end{lemma}

\begin{Proposition}\label{prop:iextraspecial}
Suppose that $q-1$ is divisible by $r$, and that $H$ is a primitive soluble subgroup of $\GL_{r^m}(q)$ in Class $\mC_6$. Let $E$ be the extraspecial type normal subgroup of $H$. Then 
\[
I(H) \geq \begin{cases}
m &\text{if } E \cong 2_{-}^{1+2m}\\
m+1 &\text{otherwise}
\end{cases} 
.\] 
\end{Proposition}

\begin{proof}
We use the matrices 
defined above. Let $\{e_1, \dots, e_r\}$ denote the standard basis of $\F_q^r$. Then $e_i$ is an eigenvector of $x$ with eigenvalue $\omega^{i-1}$, so for $i \in \{1, \dots, m\}$ and $k_1, \dots, k_m \in \{1, \dots, r\}$, the vector $e_{k_1} \otimes \dots \otimes e_{k_m}$ is an eigenvector of $x_i$ with eigenvalue $\omega^{k_i -1}$.  For $i \in \{0, \ldots, m\}$, we let
$
w_i = e_1 \otimes \dots \otimes e_1 \otimes e_2 \otimes \dots \otimes e_2 \in \F_q^r \otimes \dots \otimes \F_q^r
 $
(with $m-i$ terms $e_1$ and $i$ terms $e_2$).

We start by showing that  $I(E)$ satisfies the given bound. 
First suppose that 
$E$ is isomorphic to $r_+^{1+2m}$ or $2^{2+2m}$. Then by Lemma~\ref{lemma:propxi}, up to conjugacy in $\GL_d(q)$ the group $E$ is equal to $\langle x_1, \ldots, x_m, y_1, \ldots,  y_m\rangle$  or $\langle x_1, \dots, x_m, y_1, \dots, y_m, z\rangle$, respectively. 
Thus in both cases $E_{w_0} = \langle x_1, \ldots, x_m\rangle \cong C_r^m$. For $i \in \{1, \ldots, m\}$,  the element $x_{i} \in E_{(w_0, \dots, w_{i-1})} \setminus E_{(w_0, \dots, w_i)}.$ 
Since $E_{w_0} \cong C_r^m$, we obtain $E_{(w_0, \dots, w_m)}= 1$ and hence $(w_0, \dots, w_m)$ is an irredundant base for $E$.  
If instead $E \cong 2^{1+2m}_{-}$, so that $E = \langle x_1', x_2, \ldots, x_m, y_1',
y_2 \ldots  y_m \rangle$, then the same argument shows that $E_{w_0} \subseteq \langle x_2, \dots, x_m\rangle$. Hence the sequence $(w_0, w_2, \ldots, w_m)$ is an irredundant base for $E$.

Finally, it is immediate from Lemma~\ref{lem:boundnormalsubgroup} that $I(H) \ge I(E)$, so the result follows. 
\end{proof}

\subsection{Tensor decomposable groups}\label{subsec:tensorinduced}
 
We now consider tensor products of $\mC_6$ groups, so let
$d = r_1^{m_1} \cdots r_l^{m_l}$ be the prime factorisation of~$d$, and assume that $r_1, \dots, r_l$ divide $q-1$, so that for each $i$ the group $\GL_{r_i^{m_i}}(q)$ contains an absolutely irreducible  group of extraspecial type. Recall that $H$ denotes a soluble subgroup of $\GL_d(q)$.
%



\begin{Proposition}\label{prop:tensorproductcase}
For $1 \le i \le l$, let $E_i$ be an absolutely irreducible subgroup of $\GL_{r_i^{m_i}}(q)$ of extraspecial type, and suppose that
$H$ has a normal subgroup $T:=E_1 \otimes \cdots \otimes E_l$ such that each $E_i/Z(E_i)$ is a completely reducible $H$-module.
Then $I(H) \leq 1 + 6.49 \log d$,  and conversely \[
I(H) \geq I(T) \geq 1 + \sum_{j = 1}^l (I(E_j) -1) \geq \sum_{j = 1}^l m_j = \Omega(d).
\]
\end{Proposition}

\begin{proof}
Noting that the groups~$E_i$ have coprime orders, we write  $T = E_1 \times  \cdots \times E_l$, acting on the decomposition $V = W_1 \otimes \cdots \otimes W_l$.  We first prove the upper bound. For $1 \le i \le \ell$ let  
$N_i$ be the group induced by $H$ on $W_i$, so that 
$H \leq N_1 \otimes  \dots \otimes N_l$.
Proposition~\ref{prop:boundextraspecialcase} bounds $|N_i/ (N_i \cap Z(\GL_{r_i^{m_i}}(q)))| \leq r_i^{6.49 m_i}$ for all $i$, so 
\[ I(H) \leq 1 + \log |H/(H \cap Z)| \leq 1 + \sum_{i = 1}^l \log |N_i/(N_i \cap Z)| \leq 1 + 6.49 \sum_{i = 1}^l m_i \log r_i = 1 + 6.49 \log d, \]
and the upper bound follows.

Turning now to the lower bound, the first inequality is clear from Lemma~\ref{lem:boundnormalsubgroup}. 
For $j \in \{1, \ldots, l\}$, let $(w_1^j, \dots, w_{I(E_j)}^j)$ be an irredundant base of maximal size for the action of $E_j$ on $W_j$, and for  $t \in \{1, \ldots, I(E_j)\}$, let 
\[
v_{jt} 
= w_1^1 \otimes \dots \otimes w_1^{j-1} \otimes w_t^j \otimes w_1^{j+1} \otimes \dots \otimes w_1^l.
\] We claim that $S  = (v_{11}, \ldots, v_{1 I(E_1)}, v_{22}, \dots, v_{2 I(E_2)}, v_{32}, \ldots,  v_{l I(E_l)})$ is an irredundant sequence  for $T$. 

For all $j \in \{1, \dots, t\}$ and $t \in \{1, \ldots,  I(E_j)\}$, since $E_j \neq 1$ there exists an $h_j \in E_j$ that stabilises $v_1^j, \ldots, v_{t-1}^j$, but not $v_t^j$. Then $(1, \dots, 1,h_j,1, \dots,1)$ stabilises $v_{11}, \dots, v_{j(t-1)}$, but not~$v_{jt}$. This shows that 
$T > T_{v_{11}} >  \dots > T_{(S)}$.
Since every irredundant sequence can be extended to an irredundant base, the group~$T$ has an irredundant base of size at least $|S| = I(E_1)+ \sum_{i = 2}^l (I(E_j) -1)$.  Proposition~\ref{prop:iextraspecial} proves that $I(E_j) \geq  m_j$ if $r_j = 2$ and $I(E_j) \geq  m_j +1$ otherwise.  
Since the $r_j$ are coprime, without loss of generality $r_j > 2$ for $j > 1$ so
\[I(T) \geq  
I(E_1) + \sum_{j = 2}^l(I(E_j) - 1) \geq \sum_{j = 1}^l  m_j.\qedhere\] 
\end{proof}

\subsection{Proof of Theorem~\ref{theo:main}}\label{sec:proofmain1}

In this subsection, we first collect results on the structure of primitive linear groups, then prove Theorem~\ref{theo:main} for primitive groups, and finally prove it for imprimitive groups. 

\begin{lemma}\label{lemma:prelimsemilinear}
Suppose that $H$ is a primitive subgroup of $\GL_d(q)$, and let~$A$ be an abelian normal subgroup of~$H$.
 Then the ring $\F_qA$ is a field $\F_{q^a}$ for some divisor $a$ of $d$.
Furthermore, there is a natural embedding $\psi$ of $C:= C_{H}(A)$ into $\GL_e(q^a)$, where $ea = d$, such that $\psi(A) \le Z(\GL_e(q^a))$,  and $H \leq C.\mathrm{Gal}(\F_{q^a}:\F_q)$.
\end{lemma}

\begin{proof}
Since  $H$ is primitive, $A$ acts homogeneously on $V = \F_q^d$,  so all claims except the final one are as stated in \cite[Lemma~1.10]{LucchiniMenegazzoMorigi01}.
Finally, since $\F_q A$ is a field, the group  $H/C$ embeds into $\operatorname{Gal}(\F_{q^a} : \F_q)$.  
\end{proof}

\begin{lemma}\label{lem:prim_not_semilin} 
    Suppose that $H$ is  a primitive subgroup of $\GL_d(q)$ that is not semilinear, and that the Fitting subgroup  $F:= F(H)$ of~$H$ is absolutely irreducible. 
    Then $F = (H \cap Z) \circ E_{r_1} \otimes \cdots \otimes E_{r_l}$, where each $E_{r_i}$ is an absolutely irreducible extraspecial $r_i$-subgroup of $\GL_{r_i^{m_i}}(q)$ for a distinct prime $r_i$, of exponent $r_i$ if $r_i$ is odd, and $d = r_1^{m_1} \cdots r_l^{m_l}$.
\end{lemma}

\begin{proof}
By Lemma~\ref{lemma:prelimsemilinear}, since $H$ is not semilinear, every abelian normal subgroup of $H$ is scalar. 
    By \cite[Corollary~2.4.5]{SHO92} each Sylow subgroup $O_r(F)$ of $F$ is either 
    cyclic or the central product of a cyclic group and an extraspecial $r$-group, of exponent $r$ if $r$ is odd. Since each $O_r(F)$ is characteristic in $H$, each cyclic 
    central factor of $F$
 is scalar. Let $r_1, \ldots, r_l$ be the primes for which $O_{r_i}(F)$ is not cyclic, and let $E_{r_1}, \ldots, E_{r_l}$ be the corresponding extraspecial $r_i$-groups. 
    Then each $E_{r_i}$ is normal in $H$, so acts homogeneously, and by assumption acts absolutely irreducibly on each of its  (pairwise isomorphic) irreducible constituents. As in Subsection~\ref{subsec:c6}, this action embeds $E_{r_i}$ in 
    $\GL_{r_i^{m_i}}(q)$, as required. 
    \end{proof}

The structure of maximal primitive soluble subgroups of $\GL_d(q)$ is given by the following result.

\begin{lemma}\label{lemma:prelimdecomposition}
Let $M \leq \GL_d(q)$ be a maximal  primitive soluble group, let $F:= F(M)$ be the Fitting subgroup of $M$.
Then $M$ has a unique maximal abelian normal subgroup $A$, and  $A \cong \F_{q^a}^\ast$ for some divisor $a$ of $d$.  Under the embedding $\psi$ from Lemma~\ref{lemma:prelimsemilinear},  the group $\psi(F)$ is absolutely irreducible. Let $r_1, \dots, r_l$ be the primes for which $O_{r_i}(F)$ is not cyclic. Then 
$\psi(F)$ is a tensor product of extraspecial groups of order $r_i^{1+2m_i}$ and scalars, 
 $d/a = r_1^{m_1} \cdots r_l^{m_l}$, and $C_M(A)/F$ is isomorphic to a  subgroup of the direct product of completely reducible subgroups of symplectic groups $\Sp_{2m_i}(r_i)$.
\end{lemma}

\begin{proof}
The uniqueness of $A$ follows from \cite[Theorem~2.5.13]{SHO92}, and its structure from Lemma~\ref{lemma:prelimsemilinear}. The action of $C_M(A)/F$ is from \cite[Corollary 2.4]{Wolf82}, and 
the rest is extracted from~\cite[Theorem~2.5.19]{SHO92}.
\end{proof}
 
\begin{proof}[Proof of Theorem~\ref{theo:main} for primitive groups]
Suppose that $H$ is primitive. We first prove the upper bound on $I(H)$. There exists at least one maximal primitive soluble group $M$ such that $H \le M \le \GL_d(q)$. Let $A$ be a maximal abelian normal subgroup of $M$ and $C = C_M(A)$. 
Then by Lemma~\ref{lemma:prelimdecomposition} the group $A$ 
induces an embedding $\psi$ of $C$ into $\GL_{e}(q^a)$ for some $a$ and $e$ with $d = ae$ such that $\psi(F)$ is  absolutely irreducible.  Write $e = r_1^{m_1} \cdots r_{l}^{m_l}$. Then
$\psi(F)$ is  a tensor product of~$\psi(A)$ and absolutely irreducible extraspecial groups $E_{r_i}$  of dimension $r_i^{m_i}$, and each $E_{r_i}/Z(E_{r_i})$ is a completely reducible $C$-module. Since $F \unlhd C$, either $e =1$ and $F \cong \F_{q^d}^\ast$ or 
Proposition~\ref{prop:tensorproductcase} applies, so in both cases $I(C) \leq 1 + 6.49\log e$. 
By Lemma~\ref{lemma:prelimsemilinear} the group $M/C$ is isomorphic to a subgroup of $\operatorname{Gal}(\F_{q^a}:\F_q)$, so $\ell(M/C) \leq  \Omega(a)$ and
\[
I(H) \leq I(M) \leq 
I(C) + \Omega(a) \leq  1+ 6.49 \log e + \log a  \leq 1 + 6.49(\log e  + \log a ) = 1 + 6.49\log d.\]

For the lower bound, 
suppose that $H$ is isomorphic to $\GamL_1(q^d)$, fixing an  identification of $\F_{q^d}$ with $\F_q^d$.
Let $t = \Omega(d)$ and let $f_1, \dots, f_t$
be the prime divisors of $d$, with multiplicities. 
Let $\phi$ denote the automorphism $\xi \mapsto \xi^q$ of $\F_{q^d}$. Let $v_0 = 1$, for $i \in \{1, \dots, t\}$ choose 
$v_i \in \F_{q^{f_1\cdots f_{i}}} \setminus \F_{q^{f_1\cdots f_{i-1}}}$, 
and let $B = (v_0, \dots, v_t)$. For all $i \in \{0, \dots, t\}$, the group $H_{(v_0, \ldots, v_i)} \cong \langle \phi^{f_1 \cdots f_i} \rangle$. In particular, this yields $H >  H_{(v_0)} > \dots > H_{(v_0, \dots, v_{t})} = 1$. Thus $B$ is an irredundant base of size $t+1 = \Omega(d) +1$ for $H$.

If $H$ is not semilinear and $F(H)$ is absolutely irreducible, then the structure of $H$ is given by Lemma~\ref{lem:prim_not_semilin}, so Proposition~\ref{prop:tensorproductcase} shows that $I(H) \geq \Omega(d)$. 
\end{proof}

In fact, the assumption that $H$ is not semilinear and $F(H)$ is absolutely irreducible could be weakened by a more careful analysis, but we omit the details. 
Instead we move on to the remaining irreducible soluble groups: those that are imprimitive. 
It is clear that $I(H) \leq d$, so the next lemma completes the proof of Theorem~\ref{theo:main}.

\begin{lemma}\label{lemma:imprimitive}
Suppose that $H \leq \GL_d(q)$ is a maximal soluble imprimitive group, 
preserving a decomposition $\F_q^d = V_1 \oplus \dots \oplus V_k$, and let $m = d/k$. 
Write $H =  L \wr T$ with $L \leq \GL_m(q)$ and $T \leq \Sn_k$. Then $I(H) \geq k I(L)$. 
Hence for all $q > 2$ there exists at least one such $H$ for which $I(H)= d$. 
\end{lemma}

\begin{proof}
Let $(v_1, \ldots, v_{I(L)})$ be an irredundant base for the action of $L$ on $\F_q^m$. For $j \in \{1, \dots, k\}$, write $(v_1^j, \dots, v_{I(L)}^j)$ for a corresponding base in $V_j$. Then \[
 (v_1^1, \dots, v_{I(L)}^1, v_1^2, \dots, v_{I(L)}^2, v_1^3, \dots, v_{I(L)}^k)
\]is an irredundant base for the action of $L^k \leq H$ on $\F_q^d$, so Lemma~\ref{lem:boundnormalsubgroup} gives  $I(H) \ge I(L^k) \geq kI(L)$. 

For the final claim, let $T \leq \Sn_d$ be the cyclic group $C_d$ acting transitively, and consider the group $H = \GL_1(q) \wr T \leq \GL_d(q)$, which is soluble, and is irreducible for all $q>2$. Then $I(H) = d$. 
\end{proof}

\section{Greedy bases}\label{Greedy}

In~\cite{Seress96}, Seress conjectured that the greedy base size of each primitive soluble permutation
group is bounded above by 4. In this section we show that this is correct for primitive groups of odd order, but false in general.  More precisely, we show that if $G$ is a primitive group of 
odd order then $b(G) = \mathcal{G}(G)$, and we construct a primitive soluble group with greedy base size greater than 4.

We begin with a key ingredient in our argument for odd order groups. 

\begin{lemma}\label{subdirect}
Let $p$ be an odd prime and let $l$ and $k$ be positive integers. For $i \in \{1, \ldots, k\}$ let $V_i = \F_p^l$, and let $V = \oplus_{i = 1}^k V_i$. Let $L\leq \Gamma\mathrm{L}_1(p^l) \leq \GL_l(p)$ and $T\leq  \Sn_k$ be groups of odd order with $T$ transitive, and let $H$ be a subgroup of $L\wr T \le \GL(V)$. 
Let $v$ be a point in a largest $H$-orbit. Then there exists a $u_i\in V_i \setminus \{0\}$ for each $i$ such that $H_v\leq K:= L_{u_1} \times \cdots \times L_{u_k}.$
\end{lemma}

\begin{proof}
Let $\mO\ne\{0\}$ 
be an $L$-orbit on $\mathbb{F}_p^l$. Since $L$ has odd order, $\mO$ has an odd number of points, whence $\mO \ne-\mO = \{-x : x \in \mO\}$. Moreover, since the linear map $x\mapsto -x$ commutes with $L$, we deduce that $-\mO$ is an $L$-orbit, and moreover, $L_{x}=L_{-x}$ for all $x\in \mO$. Let $P_1$ and $P_2$ be unions of the non-zero $L$-orbits such that $P_1=-P_2$ and $P_1\sqcup P_2=\mathbb{F}_p^l\setminus\{0\}$. There exist $v_i\in V_i$ such that $v=v_1+v_2+\cdots+v_k$.

We now define a certain vector $u$ and describe $H_{u}$. By~\cite[Corollary 1]{Gluck83} there is a partition of $\{1,2,\dots,k\}$ into two parts $Q_1$ and $Q_2$ such that $T_{Q_1}=1$. For each $i \in \{1, \ldots, k\}$, first let $j \in \{1, 2\}$ be such that $i \in Q_j$, then if $v_i = 0$ let $u_i$ be any vector in $V_i \cap P_j$, otherwise let $$u_i=\begin{cases} v_i&\text{if $v_i\in P_j$,}\\ -v_i&\text{otherwise.}\end{cases}$$ Finally, let $u=u_1+\cdots+u_k$. 
Now let $h \in H_{u}$ induce $\sigma \in T$. Then by construction of $u$, the permutation $\sigma$ preserves the partition $Q_1\sqcup Q_2$, whence $\sigma =1$. Therefore, $H_u\leq H \cap K$. 

We conclude by showing that $H_{v}=H_{u}$.
 For each $i \in \{1, \ldots, k\}$, if $v_i=0$, then trivially $v_i^{L_{u_i}}=\{v_i\}$; otherwise $u_i=\pm v_i$, but $L_{v_i}=L_{-v_i}$ so again $v_i^{L_{u_i}}=v_i^{L_{v_i}}=\{v_i\}$. Therefore 
 $K$ fixes $v = v_1 + \ldots + v_k$,  and so  $H \cap 
 K\le H_{v}$. Since $v$ is in a largest $H$-orbit,  $|H_v|\leq |H_u|$, so the result follows. 
\end{proof}

\begin{proof}[Proof of Theorem~\ref{oddgreedy}]
By~\cite[Theorem 1.3]{Seress96} we can bound $b(G)\leq 3$. If $b(G)\leq2$ then $b(G) = \mG(G)$, so suppose that $b(G)=3$. 
The 
point stabiliser $H = G_0$ can be identified with an irreducible subgroup of $\mathrm{GL}(V)$ for some $V = \F_p^d$ with $p$ prime, and 
$b(H)=2$, so it suffices to show that $\mathcal{G}(H)=2$. The proof of~\cite[Theorem 1.3]{Seress96} shows that for some divisor $l$ of $d$ and some odd order $L \leq \GamL_1(p^l) \le \GL_l( p)$ and odd order transitive permutation group $T$ on $k:= d/l$ points we can embed  
$H\leq L\wr T \le \GL(V)$.

Let $v\in V$ be a vector in a maximum  size $H$-orbit. 
By Lemma~\ref{subdirect}, 
$H_v\leq K$. 
Now, for each non-zero $x \in  \F_p^l$ the group $\Gamma \mathrm{L}_1(p^l)_x$ has a regular orbit and hence $K$ 
does as well. Thus $H_v$ has a regular orbit, so $\mG(H)  = 2$. 
\end{proof}

We now conclude the paper by constructing a counterexample to Seress' conjecture.
\begin{proof}[Proof of Theorem~\ref{counterexample}]
Let $\Gamma = \GamL_1(4)$, let $P = \Sn_4 \wr \Sn_3 \le \Sn_{12}$ be imprimitive with  blocks $\mB$, and let $G \le \mathrm{A\Gamma L}_{12}(4) \le \AGL_{24}(2)$ have point stabiliser $H := \Gamma \wr P$. It is clear that $G$ is soluble, and that $H$ acts irreducibly on $\F_2^{24}$ so that $G$ is primitive. 

We analyse the possible first choices of an $H$-orbit representative by the greedy algorithm on $H$. We shall break down the coordinates of each $v \in \F_4^{12}$ into three size four \emph{chunks} corresponding to $\mB$. 
Consider first the vectors $v$ with no chunk with more than one $0$.
 The group $\Gamma$ is transitive on the non-zero elements of $\mathbb{F}_4^1$, so each chunk of $v$ can be mapped by $H$ to $v_1:=(1,1,1,0)$ or $v_2:=(1,1,1,1)$.
 Writing $L$ for the subgroup of $\Gamma$ generated by the natural Frobenius automorphism, we find that $$H_{v} \cong \begin{cases}
    ((L\wr \Sn_3)\times\Gamma)\wr \Sn_3&\text{if $v \in (v_1,v_1,v_1)^H$,}\\
    (((L\wr \Sn_3)\times\Gamma)\wr \Sn_2)\times(L\wr \Sn_4)&\text{if $v\in(v_1,v_1,v_2)^H$,}\\
    (L\wr \Sn_4\wr \Sn_2)\times((L\wr \Sn_3)\times\Gamma)&\text{if $v\in(v_2,v_2,v_1)^H$,}\\   
    L\wr\Sn_4\wr \Sn_3 &\text{if $v  \in  (v_2,v_2,v_2)^H$,}\\    
\end{cases}$$
and in particular  $|H_v|\geq 63700992$. 

Consider instead $w=(v_2,v_1,(0,0,1,1))$. We calculate $$H_{w}\cong(L\wr \Sn_4)\times((L\wr  \Sn_3) \times \Gamma)\times((L\wr \Sn_2)\times(\Gamma \wr \Sn_2)),$$ and so $|H_w|=63700992$. Thus either the greedy algorithm is permitted to choose $w$, or there is some vector with smaller stabiliser in $H$, and this vector necessarily has at least one chunk containing at least two $0$s.  Therefore, there is some vector $u$ in a largest $H$-orbit such that $H_u$ contains a subgroup which acts $K:=\Gamma \wr \Sn_2$ on some subspace $U = \F_4^2$ of a chunk.

Let $u_1 = (\alpha, \beta), u_2 = (\gamma, \delta) \in U$: we shall use the 2-transitivity of $\Gamma$ on the non-zero vectors of $\mathbb{F}_4^1$ to show that $K_{(u_1, u_2)} \neq 1$. First, if at least one coordinate entry, say $\alpha$, is zero, then letting $g_1$ be any non-trivial element of $\Gamma_{\gamma}$ we see that $1 \neq (g_1, 1) \in K_{(u_1, u_2)}$. Hence we may assume that $\alpha\beta\gamma\delta \neq 0$.
Next, if $\alpha = \gamma$, then $1 \neq \Gamma_{\alpha} \times 1 \le K_{(u_1, u_2)}$. The case $\beta = \delta$ is identical. 
Finally,  by $2$-transitivity there exists $g_2 \in \Gamma$ such that $(\alpha^{g_2}, \gamma^{g_2}) = 
(\beta, \delta)$,  and also $g_3 \in \Gamma$ such that $(\beta^{g_3}, \delta^{g_3}) = (\alpha, \gamma)$. Now let $h = (g_2, g_3)(1 \ 2) \in K$. Then $u_1^h = (\alpha^{g_2}, \beta^{g_3})^{(1 \ 2)} = u_1$ and $u_2^h = u_2$, so again $K_{(u_1, u_2)} \neq 1$.
%
%
This shows that $b(K) > 2$, and hence $\mG(H_u) \ge b(H_u) \ge b(K) \geq 3$. Therefore $\mathcal{G}(G) = \mathcal{G}(H) +1\geq 5$, and the result follows.
\end{proof}

\bibliographystyle{plain}
{\small\bibliography{refs}}

\parindent 0pt
(S.~Brenner) Department of Mathematics, TU Darmstadt, Darmstadt, Germany. 
\emph{Present address:} Department of Mathematics and Natural Sciences, University of Kassel, Kassel, Germany. \\
\emph{E-mail:} \texttt{sofia.brenner@uni-kassel.de}
\bigskip

(C.~del~Valle) School of Mathematics and Statistics, The Open University, Milton Keynes, United Kingdom. \\
\emph{E-mail:} \texttt{Coen.del-Valle@open.ac.uk}
\bigskip

(C.~M.~Roney-Dougal) School of Mathematics and Statistics, University of St Andrews, St Andrews, United Kingdom. \\
\emph{E-mail:} \texttt{colva.roney-dougal@st-andrews.ac.uk}


\end{document}